\documentclass[12pt]{amsart}
\usepackage{amsfonts}
\usepackage{amsmath}
\usepackage{amssymb}
\usepackage[margin=1.2in]{geometry}
\allowdisplaybreaks[4]

\newtheorem{theorem}{Theorem}[section]
\newtheorem{proposition}[theorem]{Proposition}

\newtheorem{lemma}[theorem]{Lemma}

\theoremstyle{definition}
\newtheorem{definition}[theorem]{Definition}

\newtheorem{remark}[theorem]{Remark}
\newtheorem{Conjecture}[theorem]{Conjecture}

\numberwithin{equation}{section}
\numberwithin{theorem}{section}

\numberwithin{equation}{section}

\begin{document}
\title[On almost-prime $k$-tuples]{On almost-prime $k$-tuples}

\author[Bin Chen]{Bin Chen}
\thanks{B. Chen gratefully acknowledges support by the China Scholarship Council (CSC)}
\address{B. Chen\\ Department of Mathematics: Analysis, Logic and Discrete Mathematics\\ Ghent University\\ Krijgslaan 281\\ B 9000 Ghent\\ Belgium}
\email{bin.chen@UGent.be}

\begin{abstract} Let $\tau$ denote the divisor function  and $\mathcal{H}=\{h_{1},...,h_{k}\}$ be an admissible set. We prove that there are infinitely many $n$ for which the product $\prod_{i=1}^{k}(n+h_{i})$ is square-free and $\sum_{i=1}^{k}\tau(n+h_{i})\leq \lfloor \rho_{k}\rfloor$, where $\rho_{k}$ is asymptotic to $\frac{2126}{2853} k^{2}$. It improves a previous result of M. Ram Murty and A. Vatwani, replacing $2126/2853$ by $3/4$. The main ingredients in our proof are the higher rank Selberg sieve and  Irving-Wu-Xi estimate for the divisor function in arithmetic progressions to smooth moduli.
\end{abstract}

\maketitle

\section{Introduction}
We consider a set $\mathcal{H}=\{h_{1},...,h_{k}\}$ of distinct non-negative integers. We call such a set is admissible if, for every prime $p$, the number of distinct residue classes modulo $p$ occupied by $h_{i}$ is less than $p$. The following conjecture is one of the greatest open problems in prime number theory.

\begin{Conjecture} [Prime $k$-tuples conjecture] \label{KTC} Given an admissible set $\mathcal{H}=\{h_{1},...,h_{k}\}$, there are infinitely many integers $n$ for which all $n+ h_{i}$ are prime.
\end{Conjecture}

The twin prime conjecture follows immediately from this by taking $\mathcal{H}=\{0, 2\}$. Although the prime $k$-tuples conjecture for $k \geq 2$ is still wide open, many mathematicians succeeded in making partial progress in various directions. One of these directions is the existence of small gaps between primes. In 2013, Zhang \cite{Z2014} showed
$$
\liminf_{n \rightarrow \infty} (p_{n+1} - p_{n}) < 7 \times 10^{7}
$$
by a refinement of the GPY method. The main ingredient of his proof is a stronger version of the Bombieri-Vinogradov theorem that is applicable when the moduli are smooth numbers. After Zhang's breakthrough, a new higher rank version of the Selberg sieve was developed by Maynard \cite{J2015} and Tao. This provided an alternative way of proving bounded gaps between primes, but had several other consequences as well since it was more flexible and could show the existence of clumps of many primes in intervals of bounded length (cf. {\cite[Theorem 1.1]{J2015}}). It is worth mentioning that Zhang's stronger version of the Bombieri-Vinogradov theorem with smooth moduli can be combined with the Maynard-Tao sieve to show there are clumps of primes in shorter intervals bounded length (cf. {\cite[Theorem 4(vi)]{PM2014}}). This means that a combination of both methods will yield better results than using Maynard-Tao sieve alone. For a further discussion of the progress in this direction, we refer the reader to \cite{PM2014}.

Another approximation to the prime $k$-tuples conjecture is to establish an upper bound for the expression
$$
\sum_{i=1}^{k}\tau(n+h_{i}),
$$
where $\tau$ stands for the divisor function. It is clear that the prime $k$-tuples conjecture follows if one has the upper bound $2k$ for infinitely many $n$. For large $k$, the current best result is

\begin{theorem}  [M. Ram Murty and A. Vatwani \cite{M-A2017}] \label{ThMV}There exists $\rho_{k}$ such that there are $\gg x(\log \log x)^{-1}(\log x)^{-k}$ integers $n\leq x$ for which the product $\prod_{i=1}^{k}(n+h_{i})$ is square-free and
$$\sum_{i=1}^{k}\tau(n+h_{i})\leq \lfloor \rho_{k}\rfloor.$$
For large k, we have $\rho_{k}\sim \frac{3}{4}k^{2}.$
\end{theorem}

We record previous results and methods. In 1997, Heath-Brown \cite{HB1997} obtained the above result with $\rho_{k}\sim \frac{3}{2}k^{2}$ by using Selberg sieve. In 2006, Ho and Tsang \cite{H-T2006} got $\rho_{k}\sim k^{2}$ by modifying Heath-Brown's sieve weights. In 2017, M. Ram Murty and Akshaa Vatwani \cite{M-A2017}  developed a general higher rank Selberg sieve with an additive twist to established Theorem \ref{ThMV}. We next describe in more detail which aspects determine the quality of their results.

When we use sieve methods to study the prime $k$-tuples conjecture, the primary aspect affecting the result is the sieve method itself. Roughly speaking, if a more general form of the sieve weights is used, there is more room to obtain better numerical results. As can be seen, for example, in Maynard's work \cite{J2015}. Another key aspect is how to deal with the error terms arising from the application of the sieve method. In order to control these error terms, the above results all exploit the divisor function analogue of the Bombieri-Vinogradov theorem. More precisely, let $(a,q)=1$, and set 
\begin{equation} \label{EFD}
	E(x, q, a)= \sum_{\substack{n \leq x \\ n\equiv a \, (\! \bmod q)}} \tau(n)- \frac{1}{\varphi (q)} \sum_{\substack{n \leq x \\ (n, q)=1}} \tau (n),
\end{equation} 
where $\varphi$ is the Euler totient function. Then for any $A > 0$ and any $\theta < 2/3$,
\begin{equation} \label{MTFD}
	\sum_{q \leq x^{\theta}} \max_{(a,q)=1} |E(x,q,a)| \ll_{A, \theta} \frac{x}{(\log x)^{A}}.
\end{equation}
In fact, (1.2) can be deduced from the following result. 
For $q < x^{2}$, we have for any $\epsilon > 0$, that
\begin{equation} \label{SHLR}
	 |E(x,q,a)| \ll_{\epsilon} q^{-1/4} x^{1/2 +\epsilon}.
\end{equation}
 This was proved independently by Selberg \cite [pp. 234-237] {Sebook} as well as Hooley \cite{H1957}  and Linnik; it is a consequence of the Weil bound for Kloosterman sums. The range of $\theta$ in \eqref{MTFD} determines the value of $\rho_{k}$ in Theorem \ref{ThMV}. Actually, the proof in \cite{M-A2017} gives $\rho_{k} \sim (2\theta_{0})^{-1}k^{2}$ provided \eqref{MTFD} holds for $0 < \theta < \theta_{0}$. We remark that the range $\theta < 2/3$ for \eqref{MTFD} is still best known result although it is reasonable to expect that \eqref{MTFD} should hold for all $\theta < 1$. In 2015,  A. J. Irving broke thtough the barrier 2/3 under the assumption that $q$ only has small prime factors by using the $q$-analogue of van der Corput's method. More accurately, given $\varepsilon >0$, Irving \cite{AJ2015} (or see Lemma \ref{5IrTH1.2} below) showed that, 
 \begin{equation} \label{IR}
 	 |E(x,q,a)| \ll_{\varepsilon} q^{-1} x^{1- \delta'}
 \end{equation}
 for $q< x^{2/3 + 1/246-\varepsilon}$ provided any prime factor of $q$ does not exceed $x^{\eta}$, where $\delta'$ and $\eta$ are some positive constants depending on $\varepsilon$. Quite recently, Wu and Xi \cite{W-X2021} developed a theory of arithmetic exponent pairs and used it to improve Irving's result by extending the range of $q$ to $q< x^{2/3 + 55/12756-\varepsilon}$ (see {\cite[Section 10]{W-X2021}} or Lemma \ref{12WuXiTH1.2} below). Note that $55/12756 \approx 1/231.92$.
 
 It is natural to ask whether we can combine the Irving-Wu-Xi estimate and the higher rank Selberg sieve with additive twist to improve Theorem 1.2. This is the main goal of the paper and we show

\begin{theorem} \label{NewTh}There exists $\rho_{k}$ such that there are $\gg x(\log\log x)^{-1}(\log x)^{-k}$ integers $n\leq x$ for which the product $\prod_{i=1}^{k}(n+h_{i})$ is square-free and
$$\sum_{i=1}^{k}\tau(n+h_{i})\leq \lfloor \rho_{k}\rfloor.$$
For large k, we have $\rho_{k}\sim \frac{2126}{2853}k^{2}.$
\end{theorem}
Observe that $2126/2853=0.74518... < 3/4$.

\section{Notation}
\
In this section, we recall the notation and terminology set up in \cite{A2016}. For a more detailed description, the reader is referred to \cite{M-A2017} and \cite{A2016}.

A $k$-tuple of integers $\underline{d}:=(d_{1},\cdot\cdot\cdot,d_{k})$ is said to be square-free if the product of its components is square-free. For a real number $R$, the inequality $\underline{d}\leq R$ means that $\prod_{i}d_{i} \leq R.$ The notation of divisibility among tuples is defined component-wise, that is,
$$
\underline{d}|\underline{n} \Longleftrightarrow d_{i}|n_{i}\  \mbox{for all} \   1\leq i\leq k.
$$
The notation of congruence among tuples, modulo a tuple, is also defined component-wise. On the other hand, we say a scalar $q$ divides the tuple $\underline{d}$ if $q$ divides the product $\prod_{i}d_{i}.$ When we explicitly write the congruence relation $\underline{d}\equiv \underline{e} \: (\bmod \ q)$, we mean that it holds for each component. 

A vector function is said to be multiplicative if all its component functions are multiplicative. In this context, we define the function $f(\underline{d})$ as the product of its component (multiplicative) functions, that is,
$$f(\underline{d}):=\prod_{i=1}^{k}f_{i}(d_{i}).$$
Similarly, a vector function $v(\underline{d})$ is called additive if all its components $v_{i}$ are additive, in which case, we define
$$v(\underline{d})=\sum_{i=1}^{k}v_{i}(d_{i}).$$
Some vector functions we will use are the Euler phi function, as well as the lcm and gcd functions. For example,
$$[\underline{d},\underline{e}]:=\prod_{i=1}^{k}[d_{i},e_{i}].$$
When written as the argument of a vector function, $[\underline{d},\underline{e}]$ will denote the tuple whose components are $[d_{i},e_{i}]$. The meaning of the use will be clear from the context.

We employ the following multi-index notation to denote mixed partial derivatives of a function $F(t)$ on $k$-tuples, 
$$F^{(\underline{\alpha})}(\underline{t}):=\frac{\partial^{\alpha}F(t_{1},\cdot\cdot\cdot,t_{k})}{(\partial t_{1})^{\alpha_{1}}\cdot\cdot\cdot(\partial t_{k})^{\alpha_{k}}},$$
for any $k$-tuple $\underline{\alpha}$ with $\alpha:=\sum_{j=1}^{k}\alpha_{j}.$

Given smooth functions $G$ and $H$ with compact support on $\mathbb{R}^{k}$, we define
$$
C(G,H)^{(\underline{a})}:=\int_{0}^{\infty}\cdot\cdot\cdot \int_{0}^{\infty}\left(\prod_{j=1}^{k}\frac{t_{j}^{a_{j}-1}}{(a_{j}-1)!}\right)G(\underline{t})^{(\underline{a})}H(\underline{t})^{(\underline{a})} \:\mathrm{d} \underline{t}
$$
and
$$ C(G,H)^{(\underline{a},\underline{b},\underline{c})}:=(-1)^{a+b}\int_{0}^{\infty}\cdot\cdot\cdot \int_{0}^{\infty}\left(\prod_{j=1}^{k}\frac{t_{j}^{c_{j}-1}}{(c_{j}-1)!}\right)G(\underline{t})^{(\underline{a})}H(\underline{t})^{(\underline{b})} \:\mathrm{d} \underline{t}.
$$

$\tau_{k} (n)$ represents the generalised divisor function, that is, the number of ways of writing $n$ as the product of $k$ positive integers. The number $\gamma$ denotes the Euler's constant. We use $\ll$ to denote Vinogradov's notation. We also use the convention $n \sim N$ to denote $N < n \leq 2N$. Alternatively, $f(x) \sim g(x)$ may also denote that $\lim_{x \rightarrow \infty} \frac{f(x)}{g(x)}=1$. The meaning will be clear from the context.
The greatest integer less than or equal to $x$ is denoted as $\lfloor x\rfloor.$ The dash over the sum means that we sum over $k$-tuples $\underline{d}$ and $\underline{e}$ with $[\underline{d},\underline{e}]$ square-free and co-prime to $W$. Throughout this paper, $\delta$ denotes a positive quantity which can be made as small as needed.

\section{Some hypotheses}
\par
\par

In this section,we review some of the salient features of the higher rank Selberg sieve discussed in \cite{M-A2017} and \cite{A2016}.

Given a set $S$ of $k$-tuples, $S=\{ \underline{n}=(n_{1},\cdot\cdot\cdot n_{k}) \}$, we seek to estimate sums of the form
\begin{align} \label{main sum form}
	\sum_{\underline{n}\in S}\omega_{\underline{n}} \bigg(\sum_{{\underline{d} \mid \underline{n}}}\lambda_{\underline{d}} \bigg)^{2},
\end{align}
where $\omega_{\underline{n}}$ is a `weight' attached to the tuples $\underline{n}$ and $\lambda_{\underline{d}}$ are parameters to be chosen. Throughout this section, the condition $\underline{n}\in S$ is understood to hold without being explicity stated. We impose the following hypotheses on this sum:
\vspace{5px}
\\
\indent $\text{H1.}$ If a prime $p$ divides a tuple $\underline{n}$ such that $p$ divides $n_{i}$ and $n_{j}$, with $i\neq j$, then $p$ must lie in some fixed finite set of primes $P_{0}$.
\vspace{5px}
\\
This hypothesis allows us to perform what is called the `$W$ trick'. That is, we set $W=\prod_{p<D_{0}}p,$ with $D_{0}$ depending on $S$, such that $p \in P_{0}$ implies that $p \mid W$. We then fix some tuple of residue classes $\underline{b}\ (\bmod \, W)$ with $(b_{i},W)=1$ for all $i$ and restrict $\underline{n}$ to be congruent to $\underline{b}$ in the sum we are concerned with.
\vspace{5px}
\\
\indent $\text{H2'.}$ With $W, \underline{b}$ as in $\text{H1}$, the function $\omega_{\underline{n}}$ satisfies
$$
	\sum_{\substack{\underline{d} \mid \underline{n} \\ \underline{n}\equiv \underline{b}\, (\!\bmod W)}} \omega_{\underline{n}}=\frac{X}{f(\underline{d})}+\frac{X^{*}}{f_{*}(\underline{d})}v(\underline{d})+r_{\underline{d}}
$$
for some real numbers $X$ and $X^{*}$  depending on the set $S$, where $f$ and $f_{*}$ are multiplicative and $v$ is additive. 
\vspace{5px}
\\
\indent $\text{H3.}$ With $f$ as in $\text{H2'}$, the components of $f$ satisfy
$$
f_{j}(p)=\frac{p}{\alpha_{j}}+O(p^{t}),\ \ \ \mbox{with} \ t<1,
$$
for each prime $p$ and some fixed $\alpha_{j}\in \mathbb{N}$, $\alpha_{j}$ independent of $X,k$. 
\vspace{5px}
\\
We denote the tuple $(\alpha_{1},\cdot\cdot\cdot ,\alpha_{k})$ as $\underline{\alpha}$ and the sum of the components $\Sigma_{j=1}^{k}\alpha_{j}$ as $\alpha.$
\vspace{5px}
\\
\indent $\text{H4'.}$ There exists $\varpi >0,\eta_{0} >0$ such that
$$\sum_{\substack{[\underline{d},\underline{e}]\leq X^{2/3+\varpi-\epsilon}\\d_{i},e_{i}\leq X^{\eta_{0}} \,  \forall i }}|r_{[\underline{d},\underline{e}]}|\ll\frac{X}{(\log X)^{A}},$$
for any $A>0$, $\epsilon >0$, as $X\rightarrow \infty.$ The implied constant may depend on $A$ and $\epsilon$.
\vspace{5px}
\\
\indent $\text{H5.}$ Let $v$ be as in $\text{H2'}$. For each $j$, there exists $\beta_{j}$, such that
$$\sum_{p}\frac{v_{j}(p)}{p^{1+\delta}}=\frac{\beta_{j}}{\delta}+O(1),\ \ \ \ \ \sum_{p}\frac{|v_{j}(p)|}{p^{1+\delta}} \ll \frac{1}{\delta}$$
as $\delta\rightarrow 0$.
\par
We shall choose $\lambda_{\underline{d}}$ in terms of a fixed symmetric function $F:[0,\infty)^{k}\rightarrow \mathbb{R},$ supported on the truncated simplex 
$$\Delta_{k}^{[\kappa]}(1):= \{(t_{1},\cdot\cdot\cdot,t_{k})\in [0,\kappa]^{k}:t_{1}+\cdot\cdot\cdot +t_{k}\leq 1\},\ \ \  \ \mbox{for some} \ \kappa >0,$$ as 
\begin{align} \label{test function form}
\lambda_{\underline{d}}=\mu(\underline{d})F\left(\frac{\log \underline{d}}{\log R}\right):=\mu(d_{1})\cdot\cdot\cdot \mu(d_{k})F\left(\frac{\log d_{1}}{\log R}, \cdot\cdot\cdot ,\frac{\log d_{k}}{\log R}\right),
\end{align}
where $R$ is some fixed power of $X$. If $\kappa=1$, we write $\Delta_{k}(1)=\Delta_{k}^{[1]}(1)$ for brevity.
Henceforth, we assume $D_{0}$ (and hence $W$) $\rightarrow \infty$ as $X \rightarrow \infty$. 
\section{Lemmas}

In this section we introduce some prerequisite results, some of which are quoted from the literature directly. These lemmas play an important role in the proof of our main theorem in section 5. 

Throughout this section, the big oh and little oh notation is understood to be with respect to $X \rightarrow \infty$. Moreover, the implied constants may depend on those parameters which are independent of $X$ (such as the function $f$, parameters $A, \alpha_{j}, \beta_{j}$, etc.) but not on those quantities which do depend on $X$ (such as $D_{0}, W, R$).

First recall the main result in \cite{M-A2017} which can used to deal with the main term arising from the application of the higher rank Selberg sieve with additive twist.

\begin{lemma} [M. Ram Murty and A. Vatwani {\cite[Lemma 4.2]{M-A2017}}] \label{1MVL4.2} Set $R$ to be some fixed power of $X$ and let $D_{0}=o(\log \log R)$. Let $f$ be a multiplicative vector function and $v$ be an additive vector function satisfying $\text{H3}$ and $\text{H5}$ respectively. Let $G,H$ be smooth functions with compact support. We denote
	$$G\left(\frac{\log \underline{d}}{\log R}\right):= G\left(\frac{\log d_{1}}{\log R}, \cdot\cdot\cdot ,\frac{\log d_{k}}{\log R}\right)$$
and similarly for $H$. Then
	$$
    \sideset{}{'}{\sum}_{\underline{d},\underline{e}}  \frac{\mu(\underline{d})\mu(\underline{e})}{f([\underline{d},\underline{e}])}v([\underline{d},\underline{e}])G\left(\frac{\log \underline{d}}{\log R}\right)H\left(\frac{\log \underline{e}}{\log R}\right)
    $$
	is obtained by $(\mbox{as}\  R \rightarrow \infty)$
	$$(1+o(1))\frac{c(W)}{(\log R)^{\alpha-1}}\sum_{j=1}^{k}\beta_{j}\alpha_{j}C_{j}^{*}(G,H)^{(\underline{\alpha})}+O\left(\frac{c(W)\log D_{0}}{(\log R)^{\alpha}}\right),$$
	where,
	$$C_{j}^{*}(G,H)^{(\underline{\alpha})}=C(G,H)^{(\underline{\alpha},\underline{\alpha},\underline{\alpha}+e_{j})}-C(G,H)^{(\underline{\alpha}-e_{j},\underline{\alpha},\underline{\alpha})}-C(G,H)^{(\underline{\alpha},\underline{\alpha}-e_{j},\underline{\alpha})},$$
	$$
	c(W):=\frac{W^{\alpha}}{\varphi(W)^{\alpha}},
	$$
	and the tuple $\underline{\alpha}\pm e_{j}$ is $(\alpha_{1},\cdot\cdot\cdot , \alpha_{j}\pm 1, \cdot\cdot\cdot , \alpha_{k}).$
\end{lemma}

\begin{lemma} [A. J. Irving {\cite[Theorem 1.2]{AJ2015}}] \label{5IrTH1.2} Let $E(x,q,a)$ be as in \eqref{EFD}. Suppose that $\varpi,\eta >0$ satisfy
	$$246\varpi +18\eta <1.$$
	There exists $\delta^{'} >0$, depending on $\varpi$ and $\eta$, such that for any $x^{\eta} \text{-} smooth$, square-free $q\leq x^{2/3+\varpi}$ and any $(a,q)=1$ we have
	$$
	E(x,q,a)\ll_{\varpi,\eta} q^{-1}x^{1-\delta^{'}}.
	$$
	\end{lemma}

\begin{lemma}  [J. Wu and P. Xi {\cite[Theorem 1.2]{W-X2021}}] \label{12WuXiTH1.2} Let $E(x,q,a)$ be defined as in \eqref{EFD}. Suppose that $\varepsilon >0$.
	There exist positive real numbers $\delta^{'} = \delta^{'} (\varepsilon)$ and $\eta =  \eta(\varepsilon)$, such that for any $q^{\eta} \text{-} smooth$, square-free $q\leq x^{2/3+55/12756- \varepsilon} $ and any $(a,q)=1$ we have
	$$
	E(x,q,a)\ll_{\varepsilon} q^{-1}x^{1-\delta^{'}}.
	$$
\end{lemma}

\begin{remark} The formulation is slightly different from \cite[Theorem 1.2]{W-X2021}, but it follows readily from a careful analysis of its proof \cite[Section 10]{W-X2021}.
\end{remark}

The following lemma is the analogue of \cite[Theorem 4.3]{M-A2017} and it can be viewed as a smoothed version of the higher rank Selberg sieve with additive twist. Here a smoothed version means that the sieve weights $\lambda_{\underline{d}}$ are supported on the $\underline{d}$ such that the product $\prod_{i}d_{i}$ only has prime factors less than $R^{\kappa}$.    

\begin{lemma} \label{NewLsmooth} Let $\lambda_{\underline{d}}$ be as in \eqref{test function form}. Suppose hypotheses $\text{H1}$, $\text{H2'}$, $\text{H4'}$ and $\text{H5}$. We also assume that both functions, $f$ and $f_{*}$ arising from $\text{H2'}$ satisfy $\text{H3}$ with $\alpha_{j}$ and $\alpha_{j}^{*}$ respectively. Let $R=X^{\frac{1}{2}(\frac{2}{3}+\varpi)-\delta}, \kappa=\frac{2 \eta_{0}}{2/3+\varpi}$ and $D_{0}=o(\log \log R)$. Then,
	\begin{align}
	\notag
	\sum_{\underline{n}\equiv \underline{b}\, (\!\bmod W)}\omega_{\underline{n}}\left(\sum_{\underline{d}\mid\underline{n}}\lambda_{\underline{d}}\right)^2=&(1+o(1))\frac{c(W)X}{(\log R)^{\alpha}}C(F,F)^{(\underline{\alpha})}
	\\
	\notag
	&+(1+o(1))\frac{c^{*}(W)X^{*}}{(\log R)^{\alpha^{*}-1}}\sum_{j=1}^{k}\beta_{j}\alpha_{j}^{*}C^{*}_{j}(F,F)^{(\underline{\alpha}^{*})}
    \end{align}
	where $C^{*}_{j}(F,F)^{(\underline{\alpha}^{*})}$ denotes the quantity
	$$C(F,F)^{(\underline{\alpha}^{*},\underline{\alpha}^{*},\underline{\alpha}^{*}+e_{j})}-C(F,F)^{(\underline{\alpha}^{*}-e_{j},\underline{\alpha}^{*},\underline{\alpha}^{*})}-C(F,F)^{(\underline{\alpha}^{*},\underline{\alpha}^{*}-e_{j},\underline{\alpha}^{*})},$$
	and
	$$\alpha^{*} =\sum_{j=1}^{k}\alpha_{j}^{*},\ \ \ \ \ c(W)=\frac{W^{\alpha}}{\varphi(W)^{\alpha}}, \ \ \ \ \
	c^{*}(W)=\frac{W^{\alpha^{*}}}{\varphi(W)^{\alpha^{*}}},
	$$
	the tuple $\underline{\alpha}^{*} \pm e_{j}$ is $(\alpha_{1}^{*},\cdot\cdot\cdot , \alpha_{j}^{*} \pm 1, \cdot\cdot\cdot , \alpha_{k}^{*}).$ 
	
	\begin{proof} Our proof follows the argument of \cite[Theorem 3.6]{A2016}. We expand the square, interchanging the order of summation, applying the $W$-trick and finally using $H2'$. We obtain
	$$
	X\sideset{}{'}{\sum}_{\substack{\underline{d},\underline{e}<R \\ d_{i}, e_{i} \leq R^{\kappa} \,\forall i}}\frac{\lambda_{\underline{d}}\lambda_{\underline{e}}}{f([\underline{d},\underline{e}])}
		+X^{*}\sideset{}{'}{\sum}_{\substack{\underline{d},\underline{e}<R \\ d_{i}, e_{i} \leq R^{\kappa} \,\forall i}}\frac{\lambda_{\underline{d}}\lambda_{\underline{e}}}{f_{*}([\underline{d},\underline{e}])}v([\underline{d},\underline{e}])
	+O\left( \sideset{}{'}{\sum}_{\substack{\underline{d},\underline{e}<R \\ d_{i}, e_{i} \leq R^{\kappa} \,\forall i}}| \lambda_{\underline{d}}| |\lambda_{\underline{e}}| | r_{[\underline{d},\underline{e}]}| \right).$$
    We have to analyze the two main terms. By the given choice of $\lambda_{\underline{d}}$, the first term can be treated as in Theorem 3.6 of  \cite{A2016} to be
	$$
	(1+o(1))\frac{c(W)X}{(\log R)^{\alpha}}C(F,F)^{(\underline{\alpha})}.$$
	The second term yields
	$$
	X^{*}\sideset{}{'}{\sum}_{\underline{d},\underline{e}<R}\frac{\mu(\underline{d})\mu(\underline{e})}{f_{*}([\underline{d},\underline{e}])}v([\underline{d},\underline{e}])F\Big(\frac{\log \underline{d}}{\log R}\Big)F\Big(\frac{\log \underline{e}}{\log R}\Big).
	$$
	By Lemma 4.1, this is given by
	$$
	(1+o(1))\frac{c^{*}(W)X^{*}}{(\log R)^{\alpha^{*}-1}}\sum_{j=1}^{k}\beta_{j}\alpha_{j}^{*}C^{*}_{j}(F,F)^{(\underline{\alpha}^{*})}.
	$$
	To complete the proof, we note that the choices of $R$ and $\kappa$ along with $H4'$ ensure that the error term is negligible.
	\end{proof}
	\end{lemma}

\section{Application to almost prime k-tuples}
 
We recall the definition of an admissible set.
 
\begin{definition} \label{Defadmisset} A set $\mathcal{H}=\{h_{1},...,h_{k}\}$ of distinct non-negative integers is said to be admissible if, for every prime $p$, there is a residue class $b_{p}\, (\bmod \ p)$ such that $b_{p}\notin \mathcal{H}\, (\bmod \ p)$.
\end{definition}
Throughout this section, we work with a fixed admissible set of size $k$, $\mathcal{H}=\{h_{1},...,h_{k}\}$, where $k$ is a sufficiently large integer.
 First we use the $W$-trick. Set $W=\prod_{p<D_{0}}p$, by the Chinese remainder theorem, we can find an integer $b$, such that $b+h_{i}$ is co-prime to $W$ for each $h_{i}$. We restrict $n$ to be in this fixed residue class $b$ modulo $W$. One can choose $D_{0}=\log \log \log N$, so that $W \sim (\log \log N)^{1+o(1)}$ by an application of the prime number theorem. We then consider the expressions,
\begin{align}  \label{Sum1}
	S_{1}=\sum_{\substack{n\sim N \\ n\equiv b \, (\! \bmod W)}}\left(\sum_{d_{j}\mid n+h_{j} \forall j}\lambda_{\underline{d}}\right)^{2},
\end{align}
\begin{align} \label{Sum2}
	S_{2}=\sum_{\substack{n\sim N \\ n\equiv b \, (\! \bmod W)}} \left( \sum_{j=1}^{k}\tau(n+h_{j})\right) \left(\sum_{d_{j}\mid n+h_{j} \forall j}\lambda_{\underline{d}} \right)^{2}.
\end{align}
For $\rho$ positive, we denote by $S(N,\rho)$ the quantity
$$\rho S_{1}-S_{2}.$$
The key point of our argument is to show, with an appropriate choice of $\lambda_{\underline{d}}$, that
$$
S(N,\rho)>0
$$
for all large $N$. This implies, there are infinitely many integers $n$ such that
$$\sum_{j=1}^{k}\tau(n+h_{j})\leq \lfloor \rho \rfloor,$$
where $\lfloor \rho \rfloor$ denotes the greatest integer less than or equal to $\rho$.

The asymptotic formula for $S_{1}$ was already derived in \cite[Lemma 4.2]{A2016}. We proceed to derive an asymptotic formula for $S_{2}.$

\subsection{An asymptotic formula for $S_{2}$}
\ 
\vspace{5px}
\\
We write
$$
S_{2}=\sum_{m=1}^{k}S_{2}^{(m)},\ \ \ \ \ S_{2}^{(m)}=\sum_{\substack{n\sim N \\ n\equiv b \, (\! \bmod  W)}} \tau(n+h_{m})\left(\sum_{d_{j}\mid n+h_{j} \forall j}\lambda_{\underline{d}} \right)^{2}
$$
In this subsection we obtain an asymptotic formula for $S_{2}^{(m)}$.

\begin{lemma} \label{smoothsum2} Assume $0< \varpi <\frac{55}{12756}$. Let $\varepsilon = \frac{55}{12756} - \varpi$,  $\eta=\eta (\varepsilon)$ be defined as in Lemma \ref{12WuXiTH1.2}, $\eta_{0}=\eta / 2$, $\kappa=\frac{2\eta_{0}}{2/3+\varpi}$. With $\lambda_{\underline{d}}$ chosen as in \eqref{test function form} and $R=N^{\frac{1}{2}(\frac{2}{3}+\varpi)-\delta}$, we have as $N\rightarrow \infty$,
\begin{align}
S_{2}^{(m)}:= & \sum_{\substack{n \sim N \\ n\equiv b \, (\! \bmod W)}}\tau(n+h_{m}) \left( \sum_{d_{j}\mid n+h_{j} \forall j}\lambda_{\underline{d}} \right)^{2}
\notag
\\
= & (1+o(1))\frac{W^{k-1}}{\varphi(W)^{k}}\frac{N}{(\log R)^{k}} \left(\frac{\log N}{\log R}\alpha^{(m)}-\beta_{1}^{(m)}-4\beta_{2}^{(m)} \right),
\notag
\end{align}
with
$$\alpha^{(m)}=\int_{\Delta_{k}(1)}t_{m} \Big(F^{(\underline{1}+e_{m})}(\underline{t}) \Big)^{2}\:\mathrm{d} \underline{t},$$
$$\beta_{1}^{(m)}=\int_{\Delta_{k}(1)}t_{m}^{2}\Big(F^{(\underline{1}+e_{m})}(\underline{t})\Big)^{2}\:\mathrm{d} \underline{t},$$
and
$$\beta_{2}^{(m)}=\int_{\Delta_{k}(1)}t_{m}F^{(\underline{1}+e_{m})}(\underline{t})F^{(\underline{1})}(\underline{t})\:\mathrm{d} \underline{t}.$$
\end{lemma}
\begin{proof}  Following the same argument as in    \cite[Lemma 5.10]{M-A2017}, we can show that $\text{H1}$, $\text{H2'}$, $\text{H3}$ and $\text{H5}$ holds. The variables in $\text{H2'}$ satisfy
$$X=\frac{\varphi(W)}{W^2}N\left(\log N+2\gamma -1+\sum_{p\mid W}\frac{2\log p}{p-1} \right),\ \ \ X^*=-\frac{\varphi(W)}{W^2}N,$$
$$f(\underline{d})=f_{*}(\underline{d})=\frac{\varphi(d_{m})}{d_{m}\tau(d_{m})}\prod_{p\mid d_{m}}\left(\frac{2p}{2p-1}\right)\prod_{j=1}^{k}\frac{d_{j}^{2}}{\varphi (d_{j})},$$
$$v(\underline{d})=\log d_{m}-\sum_{p\mid d_{m}}\frac{\log p}{2p-1}-\sum_{j\neq m}\sum_{p\mid d_{j}}\frac{2\log p}{p-1},$$
$$r_{\underline{d}}=E'(N,q,a)+O(d_{m}^{1/2}q^{\epsilon-1}\sqrt{N}),$$
where 
$$
q=W\prod_{j=1}^{k} d_{j},
$$ 
\begin{align}  \label{EFD'} 
E'(N,q,a)=\tau(\delta)\sum_{d \mid \delta}\frac{\mu(d)}{\tau(d)}E(N/\delta d,q',a_{d}),
\end{align}
with $\delta=(a,q), q'=q/ \delta, a_{d} \equiv   a \overline{\delta d} \,(\bmod \, q')$. Here $\overline{\delta d}$ is the inverse of $\delta d$ modulo $q'$ and $a$ is some integer depending on $b$, $m$, $\underline{d}$, and $W$. 
\vspace{5px}
\\
\indent $\text{H3}$ holds for $f$ and $f_{*}$ with
\begin{equation}  \label{proH3} \alpha_{j}=\alpha_{j}^{*}=\left\{
	\begin{array}{rcl}
		1     &      & \mbox{if}\ \  j=1,\cdot\cdot\cdot ,k,\ \ j\neq m,      \\
		2               &      & \mbox{if}\ \  j=m.
	\end{array} \right.
\end{equation}
\indent $\text{H5}$ holds for the additive function $v$ with $\beta_{j}$, given by
$$v_{j}(p)=-\frac{2\log p}{p-1}\ \ \ \mbox{for}\ \ j\neq m,\ \ \ \ \ v_{m}(p)=\log p-\frac{\log p}{2p-1},$$
and
\begin{equation}  \label{proH5}  \beta_{j}=\left\{
	\begin{array}{rcl}
		0     &      & \mbox{if} \ \  j=1,\cdot\cdot\cdot ,k,\ \ j\neq m,       \\
		1               &      & \mbox{if} \ \  j=m.
	\end{array} \right.
\end{equation}
\vspace{5px}
\\
\indent We give details to verify $\text{H4'}$. In fact, it suffices to show that for any $A>0, \epsilon >0$,
\begin{align}  \label{mainH4'} 
	\sideset{}{'}{\sum}_{\substack{[\underline{d},\underline{e}]\leq N^{2/3+\varpi-\epsilon} \\ d_{j},e_{j}\leq N^{\eta_{0}} \, \forall j}} \left|E'(N,q,a) \right|
+O\left(\sideset{}{'}{\sum}_{\substack{[\underline{d},\underline{e}]\leq N^{2/3+\varpi-\epsilon} \\d_{j},e_{j}\leq N^{\eta_{0}} \, \forall  j}}[d_{m},e_{m}]^{1/2}q^{\epsilon-1} \sqrt{N}\right) \ll \frac{N}{(\log N)^{A}},
\end{align}
where $q=W \prod_{j} [d_{j}, e_{j}]$. Denoting $\prod_{j\neq m}[d_{j},e_{j}]$ as $[\underline{d},\underline{e}]_{m}$, we have
\begin{align}
\sideset{}{'}{\sum}_{\substack{[\underline{d},\underline{e}]\leq N^{2/3+\varpi-\epsilon} \\d_{i},e_{i}\leq N^{\eta_{0}} \, \forall  j}}[d_{m},e_{m}]^{1/2}q^{\epsilon-1} &\ll \sideset{}{'}{\sum}_{[\underline{d},\underline{e}]\leq N^{2/3+\varpi} }[d_{m},e_{m}]^{1/2}q^{\epsilon-1}
\notag
\\
&\ll W^{\epsilon-1}\sum_{[d_{m},e_{m}]\leq N^{2/3+\varpi}}[d_{m},e_{m}]^{\epsilon-1/2}\sum_{\substack{[\underline{d},\underline{e}]_{m}\leq N^{2/3+\varpi}\\ [\underline{d},\underline{e}]_{m} \, \mbox{\footnotesize square-free}}}([\underline{d},\underline{e}]_{m})^{\epsilon-1}.
\notag
\end{align}
Using \cite[Proposition 3.1]{A2016} and partial summation along with the fact that the average order of $\tau_{3}(n)$ is $(\log n)^2$, we get
$$\sum_{[d_{m},e_{m}]\leq N^{2/3+\varpi}}[d_{m},e_{m}]^{\epsilon-1/2}\ll \sum_{r\leq N^{2/3+\varpi}}r^{\epsilon-1/2}\tau_{3}(r)\ll(N^{2/3+\varpi})^{\epsilon+1/2}(\log N)^2.$$
Similarly,
\begin{align}
	\sum_{\substack{[\underline{d},\underline{e}]_{m} \leq N^{2/3+\varpi}\\ [\underline{d},\underline{e}]_{m} \, \mbox{\footnotesize square-free}}}([\underline{d},\underline{e}]_{m})^{\epsilon-1}
	\ll \sum_{r\leq N^{2/3+\varpi}}r^{\epsilon-1}\tau_{3(k-1)}(r)\ll(N^{2/3+\varpi})^{\epsilon}(\log N)^{3k}.
	\notag
\end{align}
As $\epsilon$ can be arbitrarily small and $W\ll (\log \log N)^2$, we obtain,
$$\sideset{}{'}{\sum}_{\substack{[\underline{d},\underline{e}]\leq N^{2/3+\varpi-\epsilon} \\d_{j},e_{j}\leq N^{\eta_{0}} \, \forall  j}}[d_{m},e_{m}]^{1/2}q^{\epsilon-1} \sqrt{N} \ll N^{\epsilon'+(2/3+\varpi)/2}\sqrt{N},$$
for any $\epsilon'>0$. As $2/3+\varpi<1$, this term is indeed of the order of $N(\log N)^{-A}$ for any $A>0$ as required. We now only need to consider the first term of \eqref{mainH4'}. It can be bounded by
\begin{align}  \label{NEH4'} 
& \sum_{\substack{q\leq WN^{2/3+\varpi-\epsilon} \\ q\ | \prod\limits_{p\leq N^{\eta_{0}}}p}} \tau_{3k}(q) \max_{a \, (\! \bmod q)} \left | E'(N,q,a)\right|
\notag
\\
\ll & \left( \sum_{q\leq N^{2/3-\epsilon}} +\sum_{\substack{N^{2/3-\epsilon}<q\leq N^{2/3+\varpi-\epsilon} \\ q\ | \prod\limits_{p\leq N^{\eta_{0}}}p}} \right) \mu(q)^{2} \tau_{3k}(q) \max_{a \, (\! \bmod q)} \left| E'(N,q,a) \right|
\end{align}
We first deal with the first term of \eqref{NEH4'}. Note that  \eqref{EFD'} gives
$$E'(N,q,a)=\tau(\delta)\sum_{d\mid \delta}\frac{\mu(d)}{\tau(d)}E \left(\frac{N}{\delta d},\frac{q}{\delta},a_{d} \right),$$
where $\delta=(a,q).$
\\
If $\left(\frac{N}{\delta d}\right)^{2} > \frac{q}{\delta}$, we find by \eqref{SHLR} 
$$
\left| E \left( \frac{N}{\delta d},\frac{q}{\delta},a_{d} \right)  \right| \ll \left(\frac{q}{\delta}\right)^{-\frac{1}{4}} \left(\frac{N}{d \delta} \right)^{\frac{1}{2}+\frac{\epsilon}{4}} \ll q^{-\frac{1}{4}} N^{\frac{1}{2}+\frac{\epsilon}{4}} \ll \frac{N}{q},
$$
provided $q < N^{\frac{2}{3}- \frac{\epsilon}{3}}$.
\\
If $\left(\frac{N}{\delta d}\right)^{2}  \leq \frac{q}{\delta}$, we use the trivial bound to obtain
\begin{align*}
	\left| E \left( \frac{N}{\delta d},\frac{q}{\delta},a_{d} \right)  \right|  & \leq \left| \sum_{\substack{n\leq \frac{N}{\delta d} \\ n\equiv a_{d}\, (\! \bmod q \delta^{-1})}}  \tau(n)\right |+\frac{1}{\varphi(q \delta^{-1})} \left| \sum_{\substack{n \leq \frac{N}{\delta d} \\ (n,q \delta^{-1})=1}} \tau(n) \right|
	\notag
	\\
	& \leq N^{\epsilon} + \frac{N}{\delta d} \log N \frac{\log q \delta^{-1}}{q \delta^{-1}} \ll \frac{N \log ^{2} N}{q},
\end{align*}
where we used the estimate $\frac{1}{\varphi (n)} \ll \frac{\log n}{n}$. Therefore, we obtain 
$$
	\left| E' \left( N,q,a \right)  \right| \ll \tau(\delta)^{2}\frac{N \log ^{2} N}{q} \ll \tau(q)^{2}\frac{N \log ^{2} N}{q}
$$
for $q\leq N^{2/3-\epsilon}$.
\\
On the other hand, \cite[Theorem 5.9]{M-A2017} gives
$$ \sum_{q\leq N^{\theta}} \mu(q)^{2} \max_{y \leq N} \max_{a \, (\! \bmod q)} |E'(y,q,a)| \ll \frac{N}{(\log N)^{A'}}
$$
for any $A' >0$ and $\theta < 2/3$. Hence, 
by Cauchy-Schwarz, the first term of \eqref{NEH4'} is bounded by
\begin{align}  \label{Firterm(5.7)} 
&\sum_{q\leq N^{2/3-\epsilon}} \mu(q)^{2} \tau_{3k}(q) \max_{a \, (\! \bmod q)} \left|  E'(N,q,a) \right| 
\notag
\\
 \ll & \left(\sum_{q\leq N^{2/3-\epsilon}} \mu(q)^{2} \tau_{3k}(q)^{2} \tau(q)^{2}\frac{N \log ^{2} N}{q} \right)^{\frac{1}{2}} \left(\sum_{q\leq N^{2/3-\epsilon}} \mu(q)^{2} \max_{a \, (\! \bmod q)} |E'(N,q,a)| \right)^{\frac{1}{2}}
\notag
\\
\ll & \frac{N}{(\log N)^{A}}
\end{align}
for any $A>0$, as $N\rightarrow \infty$.
\\
\indent We now turn to the second term of \eqref{NEH4'}. 
In order to estimate $E'(N,q,a)$ for $N^{2/3-\epsilon}<q \leq N^{2/3+\varpi-\epsilon}$, $q \mid \prod_{p\leq N^{\eta_{0}}}p$, we consider three cases.
\\
If $d\geq N^{\epsilon / 2}$, the crude bound gives
\begin{align}  \label{case1secterm(5.7)} 
\left | E \left( \frac{N}{\delta d},\frac{q}{\delta},a_{d} \right) \right| 
& \leq \left| \sum_{\substack{n\leq \frac{N}{\delta d} \\ n\equiv a_{d}\, (\! \bmod q \delta^{-1})}}  \tau(n)\right |+\frac{1}{\varphi(q \delta^{-1})} \left| \sum_{\substack{n \leq \frac{N}{\delta d} \\ (n,q \delta^{-1})=1}} \tau(n) \right|
\notag
\\
& \ll N^{\frac{\epsilon}{4}} \left(\frac{N}{\delta d}\cdot \frac{\delta}{q}+1 \right)+\frac{\delta}{q}\cdot \frac{N}{\delta d}\cdot  N^{\frac{\epsilon}{4}}
\notag
\\
& \ll \frac{N^{1-\frac{\epsilon}{4}}}{q}.
\end{align}
If $d< N^{\epsilon/2}, \delta> N^{4\varpi}$,  we obtain by \eqref{SHLR},
\begin{align}  \label{case2secterm(5.7)} 
\left| E \left(\frac{N}{\delta d},\frac{q}{\delta},a_{d} \right) \right | & \ll \left(\frac{q}{\delta} \right)^{-\frac{1}{4}} \left(\frac{N}{\delta d} \right)^{\frac{1}{2}+\epsilon}
\ll \left(\frac{q}{\delta}\right)^{-\frac{1}{4}} \left(\frac{N}{\delta}\right)^{\frac{1}{2}+\epsilon}
 \ll q^{-\frac{1}{4}}N^{\frac{1}{2}-\varpi-4\varpi\epsilon+\epsilon}.
\end{align}
Finally, if $d< N^{\epsilon/2}, \delta\leq N^{4\varpi}$,
we have
$$\frac{q}{\delta}\leq \frac{N^{\frac{2}{3}+\varpi-\epsilon}}{\delta}\leq \left(\frac{N}{\delta d} \right)^{\frac{2}{3}+\varpi} = \left(\frac{N}{\delta d} \right)^{\frac{2}{3}+\frac{55}{12756}-\varepsilon},$$
and
$$N^{\eta_{0}}= N^{\frac{\eta}{2}} \leq \left(\frac{N^{\frac{2}{3}-\epsilon}}{N^{4\varpi}}\right)^{\eta} \leq \left(\frac{q}{\delta} \right)^{\eta}.$$
As $\epsilon$ can be made arbitrarily small. We obtain by Lemma \ref{12WuXiTH1.2}, 
\begin{align}  \label{case3secterm(5.7)} 
	\left| E \left(\frac{N}{\delta d},\frac{q}{\delta},a_{d} \right) \right|  \ll_{\varpi} \left(\frac{N}{\delta d} \right)^{1-\delta'}\frac{\delta}{q}
     \ll \frac{N^{1-\delta'}\delta^{\delta'}}{q}
	 \ll \frac{N^{1-\delta'+4\varpi\delta'}}{q},
\end{align}
where $\delta'$ is some positive number depending on $\varpi$.
\\
Hence, it follows by \eqref{EFD'}, \eqref{case1secterm(5.7)}, \eqref{case2secterm(5.7)}, and \eqref{case3secterm(5.7)} that
\begin{align}  \label{secterm(5.7)} 
& \sum_{\substack{N^{2/3-\epsilon}<q\leq N^{2/3+\varpi-\epsilon/2} \\ q\ | \prod\limits_{p\leq N^{\eta_{0}}}p}} \tau_{3k}(q) \max_{a \, (\! \bmod q)} \left| E'(N,q,a) \right|
\notag
\\
\ll & \sum_{N^{2/3-\epsilon}<q \leq N^{2/3+\varpi-\epsilon/2}}\tau_{3k}(q)\tau(\delta)^{2}\left(\frac{N^{1-\frac{\epsilon}{4}}}{q}+q^{-\frac{1}{4}}N^{\frac{1}{2}-\varpi-4\varpi\epsilon+\epsilon}
+\frac{N^{1-\delta'+4\varpi\delta'}}{q} \right)
\notag
\\
\ll & \frac{N}{(\log N)^{A}}
\end{align}
for any $A>0$, as $N\rightarrow \infty$.
\\
This concludes the verification of $\text{H4'}$.
\\
\indent As the choice of $D_{0}$ gives 
$$
\frac{\log D_{0}}{\log R}=o(1),
$$ 
we are now in a position to apply Lemma \ref{NewLsmooth}. The remaining argument is the same as \cite[Lemma 5.10]{M-A2017}. This completes the proof.
\end{proof}
\vspace{5px}

Noting that our choice of $\lambda_{\underline{d}}$ satisfies the conditions of \cite[Lemma 4.2]{A2016}, we combine the above lemma with the asymptotic formula for $S_{1}$ obtained in \cite[Lemma 4.2]{A2016}.
\vspace{5px}

\begin{lemma}  \label{lemsum1sum2} Assume $0< \varpi <\frac{55}{12756}$. Let $\varepsilon = \frac{55}{12756} - \varpi$,  $\eta=\eta (\varepsilon)$ be defined as in Lemma \ref{12WuXiTH1.2}, $\eta_{0}=\eta / 2$, $\kappa=\frac{2\eta_{0}}{2/3+\varpi}$. With $\lambda_{\underline{d}}$ chosen as in \eqref{test function form} and $R=N^{\frac{1}{2}(\frac{2}{3}+\varpi)-\delta}$, we have as $N \rightarrow \infty$, 
$$
S(N,\rho):=\rho S_{1}-\sum_{m=1}^{k}S_{2}^{(m)}=(1+o(1))\frac{W^{k-1}}{\varphi(W)^{k}}\frac{N}{(\log R)^{k}} \Big(\rho I(F)-\frac{\alpha^{*}}{(\frac{1}{2}(\frac{2}{3}+\varpi)-\delta)}+\beta_{1}^{*}+4\beta_{2}^{*} \Big),
$$
with
$$\alpha^{*}=\sum_{i=1}^{k}\alpha^{(m)}=k\alpha^{(k)},\ \ \ \beta_{1}^{*}=\sum_{i=1}^{k}\beta_{1}^{(m)}=k\beta_{1}^{(k)},\ \ \ \beta_{2}^{*}=\sum_{i=1}^{k}\beta_{2}^{(m)}=k\beta_{2}^{(k)},$$
and
$$I(F)=\int_{\Delta_{k}(1)}(F^{(\underline{1})}(\underline{t}))^{2}\:\mathrm{d}\underline{t}.$$
\end{lemma}
\vspace{7px}

\subsection{The choice of the test function}
\ 
\vspace{5px}
\\
Now, we adopt some ideas from \cite{J2015} and \cite{L-P2015} to choose a suitable smooth function $F$.
\vspace{5px}

Let $T=\frac{k}{\log \log k}$. Define the function $g\!: [0,\infty)\rightarrow \mathbb{R}$ by
\begin{equation}  \label{testfg} g(t):=\left\{
	\begin{array}{rcl}
		e^{-\frac{t}{2}}\big(1-\frac{t}{T}\big),     &      &\mbox{if} \ \ t\leq T,       \\
		0,               &      &\mbox{if}\ \  t>T,
	\end{array} \right.
\end{equation}
and the simplex set
$$
\Delta_{k}(r):= \{(t_{1},\cdot\cdot\cdot,t_{k}) \in [0, \infty)^{k} : t_{1}+\cdot\cdot\cdot +t_{k}\leq r\}.
$$
Let $h_{1}(t_{1},\cdot\cdot\cdot,t_{k}) \!: [0,\infty)^{k} \rightarrow \mathbb{R}$ be a smooth function with $| h_{1}(t_{1},\cdot\cdot\cdot,t_{k}) | \leq1$ such that
\begin{equation} \label{smfh1}  h_{1}(t_{1},\cdot\cdot\cdot,t_{k})=\left\{
	\begin{array}{rcl}
		1,    &      &\mbox{if} \ \ (t_{1},\cdot\cdot\cdot,t_{k})\in \Delta_{k}(1-\delta_{1}),       \\
		0,              &      & \mbox{if}\ \  (t_{1},\cdot\cdot\cdot,t_{k})\notin \Delta_{k}(1),
	\end{array} \right.
\end{equation}
where $\delta_{1} >0$ is a small constant to be chosen soon.
\\ 
Furthermore, we may assume that
\begin{align} \label{deh1} 
\left| \frac{\partial h_{1}}{\partial t_{j}}(t_{1},\cdot\cdot\cdot,t_{k})\right| \leq \frac{1}{\delta_{1}}+1
\end{align}
for each $(t_{1},\cdot\cdot\cdot,t_{k})\in \Delta_{k}(1)\setminus \Delta_{k}(1-\delta_{1})$ and $1\leq j \leq k$.
\\
\indent Let $h_{2}(t) \! :[0,\infty)\rightarrow \mathbb{R}$ be a smooth function with $| h_{2}(t) | \leq 1$ such that
\begin{equation} \label{smfh2}  h_{2}(t)=\left\{
	\begin{array}{rcl}
		1,     &      &\mbox{if} \ \ 0\leq t\leq T-\delta_{2},       \\
		0,               &      & \mbox{if} \ \  t>T,
	\end{array} \right.
\end{equation}
where $\delta_{2} < 1$ is a small positive constant to be chosen later. We may also assume that
\begin{align} \label{deh2} 
	| h_{2}'(t) | \leq \frac{1}{\delta_{2}}+1
\end{align}
for each $T-\delta_{2}\leq t \leq T$. 
\\
Finally, we define the function $F \! : [0,\infty)^{k} \rightarrow \mathbb{R} $ by
\begin{align} \label{testF} 
	F(\underline{t})=(-1)^{k}\int_{t_{1}}^{\infty}\cdot\cdot\cdot \int_{t_{k}}^{\infty}h_{1}(\underline{t})\prod_{j=1}^{k}h_{2}(kt_{j})g(kt_{j})\:\mathrm{d} \underline{t},\ \ \ \mbox{for} \ \underline{t}\in [0,\infty)^{k}.
\end{align}
As $h_{1}(\underline{t})\prod_{j=1}^{k}h_{2}(kt_{j})g(kt_{j})$ is a smooth function supported on $\Delta_{k}^{[\frac{T}{k}]}(1)$, we obtain that $F(\underline{t})$ is also a smooth function supported on $\Delta_{k}^{[\frac{T}{k}]}(1)$ and
\begin{equation} \label{deF} 
	F^{(\underline{1})}(\underline{t})=h_{1}(\underline{t})\prod_{j=1}^{k}h_{2}(kt_{j})g(kt_{j}).
\end{equation}
In view of Lemma \ref{lemsum1sum2}, our main goal for the remainder of this section becomes to estimate  $\alpha^{(k)}$,  $\beta_{1}^{(k)}$, $\beta_{2}^{(k)}$ and $I(F)$.
\\
\begin{remark}
The idea of choosing the derivative of the test function $F$ to be of the form \eqref{deF} is due to Maynard {\cite[Section 7]{J2015}}. Our introduction of smooth functions $h_{1}$ and $h_{2}$ here is inspired by the work of H. Li and H. Pan \cite{L-P2015}. 
\end{remark}

Before proceeding further, we mention some numerical results for integrals related to the function $g$. 
\begin{equation} \label{NRg1} 
	\int_{0}^{T} g(t)^{2} \:\mathrm{d} t =1-2(T + e^{-T} -1) T^{-2},
\end{equation}
\begin{equation} \label{NRg2} 
	\int_{0}^{T} tg'(t)^{2} \:\mathrm{d} t = \frac{1}{4} -\frac{Te^{-T} + e^{-T}-1}{2T^{2}},
\end{equation}
\begin{equation} \label{NRg3} 
	\int_{0}^{T} t g(t)^{2} \:\mathrm{d} t =1 + (6 - 4T- 2Te^{-T} -6e^{-T}) T^{-2}.
\end{equation}

\vspace{7px}

\subsection{An upper bound for $\alpha^{(k)}$}
\ 
\vspace{5px}
\\
Throughout the remainder of this article, any constants implied by the notion $O$ or $\ll$ are absolute. 
\\
We have
\begin{align} \label{ThreepdeF} 
	F^{(\underline{1}+e_{k})}(\underline{t})= & \frac{\partial F^{(\underline{1})}}{\partial t_{k}}(\underline{t})
	\notag
	\\
	= & \frac{\partial h_{1}}{\partial t_{k}}(\underline{t})\prod_{j=1}^{k}h_{2}(kt_{j})g(kt_{j})+kh_{1}(\underline{t})h_{2}'(kt_{k})g(kt_{k})\prod_{j=1}^{k-1}h_{2}(kt_{j})g(kt_{j})
	\notag
	\\
	& +  kh_{1}(\underline{t})h_{2}(kt_{k})g'(kt_{j})\prod_{j=1}^{k-1}h_{2}(kt_{j})g(kt_{j})
	\notag
	\\
	= : & I_{1} + I_{2} + I_{3}.
\end{align} 
By the definition of $h_{1}$, we find 
\begin{align} \label{Firarph1} 
	\int_{\Delta_{k}(1)}I_{1}^{2}t_{k}d\underline{t} & = \int_{\Delta_{k}(1)\backslash \Delta_{k}(1-\delta_{1})}t_{k} \left(\frac{\partial h_{1}}{\partial t_{k}}(\underline{t}) \right)^{2}\prod_{j=1}^{k}e^{-kt_{j}} \left(1-\frac{kt_{j}}{T} \right)^2h_{2}(kt_{j})^{2}\:\mathrm{d} \underline{t}
	\notag
	\\
	& \leq \left(1+\frac{1}{\delta_{1}} \right)^2\int_{\Delta_{k}(1)\backslash \Delta_{k}(1-\delta_{1})}t_{k}e^{-kt_{k}} \left(1-\frac{kt_{k}}{T} \right)^2h_{2}(kt_{k})^{2}\prod_{j=1}^{k-1} h_{2}(kt_{j})^{2}g(kt_{j})^{2}\:\mathrm{d}\underline{t}
	\notag
	\\
	& \leq \frac{1}{ke}\left(1+\frac{1}{\delta_{1}} \right)^2\int_{\Delta_{k}(1)\backslash \Delta_{k}(1-\delta_{1})}\prod_{j=1}^{k-1}h_{2}(kt_{j})^{2} g(kt_{j})^{2} \:\mathrm{d} \underline{t}.
\end{align}
In the last step we used $\max_{u\geq 0}ue^{-ku}=\frac{1}{ke}.$
Let $r=t_{1} + \cdot \cdot \cdot +t_{k}$, it follows
\begin{align} \label{Firarph2} 
	 \int_{\Delta_{k}(1)\backslash \Delta_{k}(1-\delta_{1})}\prod_{j=1}^{k-1}h_{2}^{2}(kt_{j})g^{2}(kt_{j}) \:\mathrm{d} \underline{t}
	\leq & \int_{\Delta_{k-1}(1)} \left(\int_{1-\delta_{1}}^{1}dr \right)\prod_{j=1}^{k-1}g(kt_{j})^{2}h_{2}(kt_{j})^2\:\mathrm{d} t_{1} \cdot \cdot \cdot \:\mathrm{d} t_{k-1}
	\notag
	\\
	\leq & \frac{\delta_{1}\Upsilon^{k-1}}{k^{^{k-1}}},
\end{align}
where $\Upsilon =\int_{0}^{T}g(t)^2 \:\mathrm{d}t.$
\\
We conclude that
\begin{align} \label{Firarphcon} 
\int_{\Delta_{k}(1)}I_{1}^{2}t_{k} \:\mathrm{d} \underline{t}\ll \frac{1}{k\delta_{1}}\cdot \frac{\Upsilon^{k-1}}{k^{^{k-1}}}.
\end{align}
from \eqref{Firarph1} and \eqref{Firarph2}.
\\
Regarding the upper bound for $\int_{\Delta_{k}(1)}I_{2}^{2}t_{k}\:\mathrm{d}\underline{t}$, we find

\begin{align} \label{secarphcon} 
	\int_{\Delta_{k}(1)}I_{2}^{2}t_{k}\:\mathrm{d}\underline{t}
	\leq & \left(1+\frac{1}{\delta_{2}} \right)^2\int_{\Delta_{k-1}(1)}\left(\prod_{j=1}^{k-1}g(kt_{j})^2h_{2}(kt_{j})^{2}\right)
	\notag
	\\
	& \ \ \ \ \ \ \cdot  \left(\int_{(T-\delta_{2})/k}^{T/k}t_{k}k^{2}e^{-kt_{k}}\left(1-\frac{kt_{k}}{T}\right)^2dt_{k}\right)\:\mathrm{d}t_{1}\cdot \cdot \cdot \:\mathrm{d}t_{k-1}
	\notag
	\\
	\leq & \left(1+\frac{1}{\delta_{2}}\right)^2\int_{\Delta_{k-1}(1)}\left(\prod_{j=1}^{k-1}g(kt_{j})^2h_{2}(kt_{j})^{2}\right)
	\notag
	\\
	& \ \ \ \ \ \ \cdot 
	\frac{\delta_{2}}{k}\cdot\frac{T}{k}\cdot k^{2}e^{-(T-\delta_{2})}\left(\frac{\delta_{2}}{T}\right)^2\:\mathrm{d}t_{1}\cdot \cdot \cdot \:\mathrm{d}t_{k-1}
	\notag
	\\
	\ll & \frac{\delta_{2}}{Te^T} \cdot \frac{\Upsilon^{k-1}}{k^{^{k-1}}}.
\end{align}
In the second inequality, we used the trivial bound for the second integral.
\\
We now estimate $\int_{\Delta_{k}(1)}I_{3}^{2}t_{k}\:\mathrm{d}\underline{t}$.
\begin{align} \label{Thiarphcon} 
    \int_{\Delta_{k}(1)}I_{3}^{2}t_{k}\:\mathrm{d}\underline{t} = & \int_{\Delta_{k}(1)}t_{k}k^2h_{1}(\underline{t})^{2}h_{2}(kt_{k})^2g'(kt_{k})^2\prod_{j=1}^{k-1}h_{2}(kt_{j})^{2}g(kt_{j})^2\:\mathrm{d}\underline{t}
	\notag
	\\
	\leq & k^2\int_{0}^{\infty}t_{k}h_{2}(kt_{k})^2g'(kt_{k})^2\:\mathrm{d}t_{k} \prod_{j=1}^{k-1}\int_{0}^{\infty}h_{2}(kt_{j})^{2}g(kt_{j})^2\:\mathrm{d}t_{j}
	\notag
	\\
	\leq &
	k^2\int_{0}^{\frac{T}{k}}t_{k}g'(kt_{k})^2\:\mathrm{d}t_{k}\prod_{j=1}^{k-1}\int_{0}^{\frac{T}{k}}g(kt_{j})^2\:\mathrm{d}t_{j}
	\notag
	\\
	= &
	\frac{\Upsilon^{k-1}}{k^{^{k-1}}} \int_{0}^{T}tg'(t)^2dt.
\end{align}
\\
From the Cauchy-Schwarz inequality, \eqref{Firarphcon}, \eqref{secarphcon}, \eqref{Thiarphcon}, and \eqref{NRg2} we deduce
\begin{align} \label{arphcon1} 
	 \alpha^{(k)}= & \int (I_{1}+I_{2}+I_{3})^2t_{k}\:\mathrm{d}\underline{t}
	\notag
	\\
	= & \int (I_{1}^2+I_{2}^2+I_{3}^2+2I_{1}I_{2}+2I_{1}I_{3}+2I_{2}I_{3})t_{k}\:\mathrm{d}\underline{t}
	\notag
	\\ 
	\leq & \int I_{1}^2t_{k}d\underline{t} + \int I_{2}^2t_{k}\:\mathrm{d} \underline{t} + \int I_{3}^2t_{k}\:\mathrm{d} \underline{t} + 2\left(\int I_{1}^2t_{k}\:\mathrm{d}\underline{t}\right)^{\frac{1}{2}} \left(\int I_{2}^2t_{k}\:\mathrm{d}\underline{t}\right)^{\frac{1}{2}}
	\notag
	\\ 
	 & + 
	2\left(\int I_{1}^2t_{k}\:\mathrm{d}\underline{t}\right)^{\frac{1}{2}} \left(\int I_{3}^2t_{k}\:\mathrm{d}\underline{t}\right)^{\frac{1}{2}}
	+2\left(\int I_{2}^2t_{k}\:\mathrm{d}\underline{t}\right)^{\frac{1}{2}} \left(\int I_{3}^2t_{k}\:\mathrm{d}\underline{t}\right)^{\frac{1}{2}}
	\notag
	\\ 
	\leq &
	\frac{\Upsilon^{k-1}}{k^{k-1}} \int_{0}^{T}tg'(t)^2\:\mathrm{d}t +O\left( \left( \frac{1}{k\delta_{1}}+  \sqrt{\frac{\delta_{2}}{k\delta_{1}Te^T}} +\frac{1}{\sqrt{k\delta_{1}}}  +\sqrt{\frac{\delta_{2}}{Te^T}}  \right)  \frac{\Upsilon^{k-1}}{k^{^{k-1}}}\right).
\end{align}
Now, selecting $\delta_{1}=\frac{\sqrt{\log k}}{k}$ and observing that $T= \frac{k}{\log \log k}$ and $\delta_{2} < 1$ gives that
\begin{align} \label{arphcon} 
	\alpha^{(k)} \leq \frac{\Upsilon^{k-1}}{k^{^{k-1}}} \int_{0}^{T}tg'(t)^2\:\mathrm{d}t  +O\left( \frac{1}{(\log k)^{\frac{1}{4}}} \cdot \frac{\Upsilon^{k-1}}{k^{^{k-1}}}\right).
\end{align}
\vspace{7px}

 \subsection{An upper bound for $\beta_{1}^{(k)}$ and $\beta_{2}^{(k)}$}
 \
 \vspace{5px}
 \\
 Recall 
 \begin{align}  \label{defbeta1} 
 	\beta_{1}^{(k)} =  \int_{\Delta_{k}(1)}t_{k}^{2}(F^{(\underline{1}+e_{k})}(\underline{t}))^{2}\:\mathrm{d} \underline{t}
 	= \int_{\Delta_{k}(1)}t_{k}^{2}(I_{1}+I_{2}+I_{3})^{2}\:\mathrm{d}\underline{t}.
 \end{align}
By an analogous argument as in the estimation of  $\alpha^{(k)}$, it is not hard to see the main contribution of right hand side of \eqref{defbeta1}  comes from $\int_{\Delta_{k}(1)}I_{3}^{2}t_{k}^{2}\:\mathrm{d}\underline{t}$. Hence, 
\begin{align}  \label{beta1con} 
	\beta_{1}^{(k)} \ll & \int_{\Delta_{k}(1)}I_{3}^{2}t_{k}^{2}\:\mathrm{d}\underline{t}
	\notag
	\\	
	= & \int_{\Delta_{k}(1)}t_{k}^{2}k^2h_{1}(\underline{t})^{2}h_{2}(kt_{k})^2g'(kt_{k})^2\prod_{j=1}^{k-1}h_{2}(kt_{j})^{2}g(kt_{j})^2\:\mathrm{d}\underline{t}
	\notag
	\\
	\leq & \frac{\Upsilon^{k-1}}{k^{k}} \int_{0}^{T}t^{2}g'(t)^2\:\mathrm{d}t \ll  \frac{\Upsilon^{k-1}}{k^{k}}.
\end{align}
In the last step we used the fact $\int_{0}^{T}t^{2}g'(t)^2\:\mathrm{d}t = O(1),$ as $k \rightarrow \infty$. This can be readily seen by noting that the integrand is an exponentially decreasing function.
\\
Similarly, we have 
\begin{align}  \label{defbeta2} 
	\beta_{2}^{(k)} =  \int_{\Delta_{k}(1)}t_{k}F^{(\underline{1}+e_{k})}(\underline{t})F^{(\underline{1})}(\underline{t})\:\mathrm{d} \underline{t} =\int_{\Delta_{k}(1)}t_{k}(I_{1}+I_{2}+I_{3})F^{(\underline{1})}(\underline{t})\:\mathrm{d}\underline{t}.
\end{align}
The main contribution of the right hand side of \eqref{defbeta2} comes from $\int_{\Delta_{k}(1)}t_{k} I_{3} F^{(\underline{1})}(\underline{t}) \:\mathrm{d}\underline{t}$. Therefore,
 \begin{align}  \label{beta2con} 
 	\beta_{2}^{(k)} \ll & \int_{\Delta_{k}(1)}t_{k} I_{3} F^{(\underline{1})}(\underline{t}) \:\mathrm{d}\underline{t}
 	\notag
 	\\	
 	= & \int_{\Delta_{k}(1)}t_{k}kh_{1}^{2}(\underline{t})h_{2}(kt_{k})^2g'(kt_{k})g(kt_{k})\prod_{j=1}^{k-1}h_{2}(kt_{j})^{2}g(kt_{j})^2\:\mathrm{d}\underline{t}
 	\notag
 	\\
 	\leq & \frac{\Upsilon^{k-1}}{k^{k}} \int_{0}^{T}tg'(t)g(t)\:\mathrm{d}t
 	\ll \frac{\Upsilon^{k-1}}{k^{k}}.
 \end{align}

\subsection{Lower bound for $I(F)$}
\
\vspace{5px}
\\
In this subsection we derive a lower bound for $I(F)$. Our argument is inspired by {\cite[Section 7]{J2015}}. 

\begin{proposition}  \label{proIF}  For every $\epsilon > 0$, there exists $\delta_{2}= \delta_{2}(\epsilon) >0$, such that
	\begin{align}  \label{LbIF} 
		I(F) > \frac{\Upsilon^{k-1}}{k^{k}} & \left( 1- \frac{T}{k(1-T/k-\mu)^{2}} \right)  \int_{0}^{\infty} g(u)^{2} \:\mathrm{d} u -\epsilon 
		\notag
		\\
		& + O\left( \frac{e^{\sqrt{\log k}}(\log k)^{\frac{9}{2}}}{\sqrt{k-1}} \frac{\Upsilon^{k-1}}{k^{k}} \right),
	\end{align}
where
\begin{align*}
	\mu =\frac{\int_{0}^{\infty}u g(u)^{2}\:\mathrm{d}u}{\int_{0}^{\infty} g(u)^{2} \:\mathrm{d}u} \ \ \ \ \ \mbox{and} \ \ \ \ \ \Upsilon =\int_{0}^{T}g(t)^2 \:\mathrm{d}t.
\end{align*}
\end{proposition}
\begin{proof} Recall the definition of $I(F)$ and our test function $F$, we have
	\begin{align*}
		I(F)= \int_{\Delta_{k}(1)}(F^{(\underline{1})}(\underline{t}))^{2}\:\mathrm{d}\underline{t}
		= \int_{\Delta_{k}(1)} h_{1}(\underline{t})^{2}\prod_{j=1}^{k}h_{2}(kt_{j})^{2}g(kt_{j})^{2} \:\mathrm{d}\underline{t}.
	\end{align*}
In view of Lebesgue's dominated convergence theorem, for every $\epsilon >0$, we can take $\delta_{2}= \delta_{2}(\epsilon)$ sufficiently small, such that
\begin{equation}  \label{LebIF} 
	I(F) > \int_{\Delta_{k}(1)} h_{1}(\underline{t})^{2}\prod_{j=1}^{k}g(kt_{j})^{2} \:\mathrm{d}\underline{t} -\epsilon.
\end{equation}
By the definition of $h_{1}(\underline{t})$, it follows that
\begin{align}  \label{Lebh1IF} 
	I(F) > \int_{\Delta_{k}(1)} \prod_{j=1}^{k}g(kt_{j})^{2} \:\mathrm{d}\underline{t} - \int_{\Delta_{k}(1) \backslash \Delta_{k}(1- \delta_{1})} \prod_{j=1}^{k}g(kt_{j})^{2} \:\mathrm{d}\underline{t} -\epsilon.
\end{align}
Thus, to prove \eqref{LbIF}, it suffices to establish 
\begin{align}  \label{BIF1} 
	\int_{\Delta_{k}(1)} \prod_{j=1}^{k}g(kt_{j})^{2} \:\mathrm{d}\underline{t}  \geq \frac{\Upsilon^{k-1}}{k^{k}}  \left( 1- \frac{T}{k(1-T/k-\mu)^{2}} \right)  \int_{0}^{\infty} g(u)^{2} \:\mathrm{d} u,
\end{align}
and
\begin{align} \label{BIF2} 
	\int_{\Delta_{k}(1) \backslash \Delta_{k}(1- \delta_{1})} \prod_{j=1}^{k}g(kt_{j})^{2} \:\mathrm{d}\underline{t} \ll \frac{e^{\sqrt{\log k}}(\log k)^{\frac{9}{2}}}{\sqrt{k-1}} \frac{\Upsilon^{k-1}}{k^{k}}.
\end{align}
We begin with \eqref{BIF2}. It is clear that
\begin{align} \label{BIF2sp1} 
	\int_{\Delta_{k}(1) \backslash \Delta_{k}(1- \delta_{1})} \prod_{j=1}^{k}g(kt_{j})^{2} \:\mathrm{d}\underline{t}
	= & \int_{\substack{\Delta_{k}(1) \backslash \Delta_{k}(1- \delta_{1}) \\ \underline{t} \in [0, T/k]^{k}}} \prod_{j=1}^{k} e^{-kt_{j}}\Big(1-\frac{kt_{j}}{T}\Big)^2 \:\mathrm{d}\underline{t}
	\notag
	\\
	\leq & \int_{\Delta_{k}(1) \backslash \Delta_{k}(1- \delta_{1})} e^{-k (\sum_{j=1}^{k} t_{j})} \:\mathrm{d}\underline{t}
	\notag
	\\
	\leq & e^{-k (1 - \delta_{1})} \int_{\Delta_{k}(1) \backslash \Delta_{k}(1- \delta_{1})} 1 \:\mathrm{d}\underline{t}
	\notag
	\\
	\leq & e^{-k (1 - \delta_{1})} \int_{\Delta_{k-1}(1)} \int_{1 - \delta_{1}}^{1} \:\mathrm{d} r \:\mathrm{d} t_{1} \cdot \cdot \cdot  \:\mathrm{d} t_{k-1}
	\notag
	\\
	= & \frac{e^{-k (1 - \delta_{1})} \delta_{1}}{(k-1)!}.
	\end{align}
In the penultimate step, we changed the variable by $r = t_{1} + \cdot \cdot \cdot + t_{k}$.
\\
Combining Stirling's formula with \eqref{BIF2sp1}, we obtain
\begin{equation} \label{BIF2sp2} 
	\int_{\Delta_{k}(1) \backslash \Delta_{k}(1- \delta_{1})} \prod_{j=1}^{k}g(kt_{j})^{2} \:\mathrm{d}\underline{t} \ll \frac{e^{k \delta_{1}} \delta_{1}}{\sqrt{k-1}(k-1)^{k-1}}.
\end{equation}
From \eqref{NRg1}, we find
\begin{align} \label{Lbgram} 
	\Upsilon =\int_{0}^{T} g(t)^{2} \:\mathrm{d} t \geq 1 - \frac{2}{T}.
\end{align}
As $T = \frac{k}{\log \log k}$, one has
\begin{align} \label{Lbgramcon} 
	\Upsilon^{k-1} \geq \left( 1- \frac{2 \log \log k}{k} \right)^{k-1} =  e^{(k-1) \log \left( 1- \frac{2 \log \log k}{k} \right)}
	\geq  e^{(k-1) \frac{-4 \log \log k}{k}}
	\geq  \frac{1}{(\log k)^{4}}.
\end{align}
 Combining \eqref{BIF2sp2} with \eqref{Lbgramcon} then gives, as $\delta_{1}= \frac{\sqrt{\log k}}{k}$,
\begin{align} \label{BIF2con} 
	\frac{k^{k}}{\Upsilon^{k-1}}\int_{\Delta_{k}(1) \backslash \Delta_{k}(1- \delta_{1})} \prod_{j=1}^{k}g(kt_{j})^{2} \:\mathrm{d}\underline{t} \ll &
	\frac{e^{k \delta_{1}} \delta_{1} k^{k} (\log k)^{4}}{\sqrt{k-1}(k-1)^{k-1}}
	\ll \frac{e^{\sqrt{\log k}} (\log k)^{\frac{9}{2}}}{\sqrt{k-1}},
\end{align}
and the claim \eqref{LbIF} follows.
\\
\indent Now we show \eqref{BIF1}. Since squares are nonnegative, we restrict the outer integral to $\sum_{j=2}^{k} t_{j} \leq 1- T/k$, 
\begin{align} \label{rwIF1} 
	\int_{\Delta_{k}(1)} \prod_{j=1}^{k}g(kt_{j})^{2} \:\mathrm{d}\underline{t} \geq & \idotsint \limits_{\substack{t_{2}, \cdot \cdot \cdot\ , t_{k} \geq 0\\ \sum_{j=2}^{k} t_{j} \leq 1- T/k}} \int_{0}^{T/k} \prod_{j=1}^{k} g(kt_{j})^{2} \:\mathrm{d}t_{1}  \:\mathrm{d}t_{2} \cdot \cdot \cdot  \:\mathrm{d}t_{k}  
	\notag
	\\
	= & I' - E,
\end{align}
where
\begin{align} \label{FirIF1} 
	I' = & \idotsint \limits_{t_{2}, \cdot \cdot \cdot\ , t_{k} \geq 0} \int_{0}^{T/k} \prod_{j=1}^{k} g(kt_{j})^{2} \:\mathrm{d}t_{1}  \:\mathrm{d}t_{2} \cdot \cdot \cdot  \:\mathrm{d}t_{k}  
	=  \left( \int_{0}^{\infty} g(kt)^{2} \:\mathrm{d}t \right)^{k}
	=  \frac{\Upsilon^{k}}{k^{k}},
\end{align}
\begin{align} \label{SecIF1} 
	E = & \idotsint \limits_{\substack{t_{2}, \cdot \cdot \cdot\ , t_{k} \geq 0\\ \sum_{j=2}^{k} t_{j} > 1- T/k}} \int_{0}^{T/k} \prod_{j=1}^{k} g(kt_{j})^{2} \:\mathrm{d}t_{1}  \:\mathrm{d}t_{2} \cdot \cdot \cdot  \:\mathrm{d}t_{k}  
	\notag
	\\
	= & k^{-k} \left( \int_{0}^{\infty} g(u)^{2} \:\mathrm{d}u \right)
	\idotsint \limits_{\substack{u_{2}, \cdot \cdot \cdot\ , u_{k} \geq 0\\ \sum_{j=2}^{k} u_{j} > k- T}}  \prod_{j=2}^{k} g(u_{j})^{2} \:\mathrm{d}u_{2} \cdot \cdot \cdot  \:\mathrm{d}u_{k}.
\end{align}
We can check the choice of $g$ satisfies
\begin{align} \label{promu} 
	\mu = \frac{\int_{0}^{\infty}u g(u)^{2}\:\mathrm{d}u}{\int_{0}^{\infty} g(u)^{2} \:\mathrm{d}u} < 1 - \frac{T}{k}.
\end{align}
Actually, from \eqref{NRg1} and \eqref{NRg3}, we have
\begin{align} \label{Nrmu} 
	\mu = & \frac{\int_{0}^{\infty}u g(u)^{2}\:\mathrm{d}u}{\int_{0}^{\infty} g(u)^{2} \:\mathrm{d}u} 
	=  \frac{1 + (6 - 4T- 2Te^{-T} -6e^{-T}) T^{-2}}{1-2(T + e^{-T} -1) T^{-2}}
	\notag
	\\
	= & 1 - \frac{2T + 2Te^{-T} + 4e^{-T} -4}{1-2(T + e^{-T} -1) T^{-2}}.
\end{align}
Since $T = \frac{k}{\log \log k}$, it follows
\begin{align} \label{promucon} 
	1- \mu - \frac{T}{k} = \frac{2T + 2Te^{-T} + 4e^{-T} -4}{1-2(T + e^{-T} -1) T^{-2}} - \frac{T}{k} \gg 1.
\end{align}
Let $\varTheta = (k - T)/ (k-1) - \mu >0$. If $\sum_{j=2}^{k} u_{j} > k-T$, then $\sum_{j=2}^{k} u_{j} > (k-1)(\mu + \varTheta)$, and so we have
\begin{align} \label{provartheta} 
	1 \leq \varTheta^{-2} \left( \frac{1}{k-1} \sum_{j=2}^{k} u_{j} - \mu \right)^{2}.
\end{align}
Since the right hand side of \eqref{provartheta} is nonnegative for all $u_{j}$, we can obtain an upper bound for $E$ if we multiply the integrand by $\varTheta^{-2} \left( \sum_{j=2}^{k} u_{j}/(k-1) - \mu \right)^{2}$ and then drop the requirement that $\sum_{j=2}^{k} u_{j} > k -T$. We find
\begin{align} \label{ubE} 
	E \leq \varTheta^{-2} k^{-k} \left( \int_{0}^{\infty} g(u)^{2} \:\mathrm{d}u \right) \int _{0}^{\infty}
	\cdot \cdot \cdot \int _{0}^{\infty} \left( \frac{\sum_{j=2}^{k} u_{j}}{k-1} - \mu \right)^{2} \left( \prod_{j=2}^{k} g(u_{j}) ^{2} \right) \:\mathrm{d}u_{2} \cdot \cdot \cdot  \:\mathrm{d}u_{k}.
\end{align}
Expanding out the inner square and calculating all the terms which are not of the form $u_{j}^{2}$ gives
\begin{align} \label{nondiaterm} 
	\int _{0}^{\infty}
	\cdot \cdot \cdot \int _{0}^{\infty} & \left( \frac{2\sum_{2 \leq i < j \leq k} u_{i} u_{j}}{(k-1)^{2}} - \frac{2 \mu \sum_{j=2}^{k} u_{j}}{k-1}  + \mu^{2} \right) \left( \prod_{j=2}^{k} g(u_{j}) ^{2} \right) \:\mathrm{d}u_{2} \cdot \cdot \cdot  \:\mathrm{d}u_{k}
	\notag
	\\
	 &=  \frac{k-2}{k-1} \mu^{2} \Upsilon^{k-1} - 2 \mu^{2} \Upsilon^{k-1} + \mu^{2}\Upsilon^{k-1} 
	\notag
	\\
	 &=  \frac{- \mu^{2} \Upsilon^{k-1}}{k-1}.\quad \quad \quad \quad \quad \quad \quad \quad \quad \quad \quad \quad \quad \quad \quad 
\end{align}
For the $u_{j}^{2}$ terms, we see that $u_{j}^{2}g(u_{j})^{2} \leq T u_{j} g(u_{j})^{2}$ in view of the support of $g$. Hence,
\begin{align} \label{diaterm} 
	\int _{0}^{\infty}
	\cdot \cdot \cdot \int _{0}^{\infty} u_{j}^{2} \left( \prod_{i=2}^{k} g(u_{i}) ^{2} \right) \:\mathrm{d}u_{2} \cdot \cdot \cdot  \:\mathrm{d}u_{k}
	\leq T \Upsilon^{k-2} \int _{0}^{\infty} u_{j} g(u_{j})^{2} \:\mathrm{d}u_{j}
	=  \mu T \Upsilon^{k-1}.
\end{align}
It follows from \eqref{ubE}, \eqref{nondiaterm}, and \eqref{diaterm} that
\begin{align} \label{BE1} 
	E  & \leq  \varTheta^{-2} k^{-k} \left( \int_{0}^{\infty} g(u)^{2} \:\mathrm{d}u \right) \left( \frac{\mu T \Upsilon^{k-1}}{k-1} - \frac{\mu^{2} \Upsilon^{k-1}}{k-1} \right)
	\notag
	\\
	& \leq \left( \frac{\varTheta^{-2} \mu T k^{-k} \Upsilon^{k-1}}{k-1} \right) \left( \int_{0}^{\infty} g(u)^{2} \:\mathrm{d}u \right).
\end{align}
Since $(k-1) \varTheta^{2} \geq k (1- T/k - \mu)^{2}$ and $\mu \leq 1$, from \eqref{BE1} we obtain
\begin{align} \label{BEcon} 
	E  \leq \left( \frac{T k^{-k-1} \Upsilon^{k-1}}{(1- T/k - \mu)^{2}} \right) \left( \int_{0}^{\infty} g(u)^{2} \:\mathrm{d}u \right).
\end{align}
From \eqref{rwIF1}, \eqref{FirIF1}, \eqref{BEcon} we conclude \eqref{BIF1}. The proof of the proposition is now complete.
\end{proof} 
\vspace{7px}

\subsection{Completion of the proof of Theorem 1.3}

\begin{proof} Recall that $ \mathrm{supp} \  F (\underline{t}) \subset  \Delta_{k}^{[\frac{T}{k}]}(1)$ and $T/k = 1/ \log \log k$. We can find a sequence  $\{ \varpi_{k}\}_{k=0}^{\infty} \subset (0, 55/12756)$ with $\lim_{k} \varpi_{k}=55/12756$ and a real number $K$, such that 
	\begin{align} \label{provarpi} 
		\frac{\eta (\varepsilon_{k})}{2/3+\varpi_{k}} \geq \frac{1}{\log \log k},
	\end{align}
for $k > K$, where $\varepsilon_{k}= 55/12756 - \varpi_{k}$ and the function $\eta (\cdot)$ is defined as in Lemma \ref{12WuXiTH1.2}.
Applying Lemma \ref{lemsum1sum2} with $\varpi = \varpi_{k}$, we get  $S(N, \rho) >0$ for all large $N$, provided 
\begin{align} \label{GoalTh} 
	\rho > \frac{k \alpha^{(k)}}{(\frac{1}{2}(\frac{2}{3}+\varpi_{k})- \delta) I(F)} - \frac{k \beta_{1}^{(k)}}{I(F)}-\frac{4k \beta_{2}^{(k)}}{I(F)},
\end{align}
\\
Plugging the estimates for $\alpha^{(k)}$, $\beta_{1}^{(k)}$, $\beta_{2}^{(k)}$ and $I(F)$ (see \eqref{arphcon}, \eqref{beta1con}, \eqref{beta2con}, and Proposition \ref{proIF}) into the right hand side of \eqref{GoalTh} yields
\begin{align} \label{upprhgoal1} 
	&\frac{k \alpha^{(k)}}{(\frac{1}{2}(\frac{2}{3}+\varpi_{k})- \delta) I(F)} - \frac{k \beta_{1}^{(k)}}{I(F)}-\frac{4k \beta_{2}^{(k)}}{I(F)}
	\notag
	\\
	\leq &  \frac{(\frac{1}{2}(\frac{2}{3}+\varpi_{k})- \delta)^{-1} \frac{\Upsilon^{k-1}}{k^{^{k-2}}} \int_{0}^{T}tg'(t)^2\:\mathrm{d}t  +O\left( \frac{1}{(\log k)^{\frac{1}{4}}} \cdot \frac{\Upsilon^{k-1}}{k^{^{k-2}}}\right)}{\frac{\Upsilon^{k-1}}{k^{k}}  \left( 1- \frac{T}{k(1-T/k-\mu)^{2}} \right)  \int_{0}^{T} g(t)^{2} \:\mathrm{d} t -\epsilon
		 + O\left( \frac{e^{\sqrt{\log k}}(\log k)^{\frac{9}{2}}}{\sqrt{k-1}} \frac{\Upsilon^{k-1}}{k^{k}} \right)}.
\end{align}
We choose $\delta_{2} > 0$ sufficiently small such that this estimate becomes
\begin{align}  \label{upprhgoal2} 
	  \leq
       \frac{(\frac{1}{2}(\frac{2}{3}+\varpi_{k})- \delta)^{-1} \frac{\Upsilon^{k-1}}{k^{^{k-2}}} \int_{0}^{T}tg'(t)^2\:\mathrm{d}t  +O\left( \frac{1}{(\log k)^{\frac{1}{4}}} \cdot \frac{\Upsilon^{k-1}}{k^{^{k-2}}}\right)}{\frac{\Upsilon^{k-1}}{k^{k}}  \left( 1- \frac{T}{k(1-T/k-\mu)^{2}} \right)  \int_{0}^{T} g(t)^{2} \:\mathrm{d} t
	   + O\left( \frac{e^{\sqrt{\log k}}(\log k)^{\frac{9}{2}}}{\sqrt{k-1}} \frac{\Upsilon^{k-1}}{k^{k}} \right)}.
\end{align}
Since \eqref{promucon} and $T/k = 1/ \log \log k$, we have
\begin{align}  \label{lioterm} 
	\frac{T}{k(1-T/k-\mu)^{2}} = o(1), \ \ \ \mbox{as} \ k \rightarrow \infty.
\end{align}
It follows from \eqref{NRg1} and \eqref{NRg2} that
\begin{align}  \label{asyg} 
	\frac{ \int_{0}^{T}tg'(t)^2\:\mathrm{d}t } {\int_{0}^{T} g(t)^{2} \:\mathrm{d} t} \rightarrow \frac{1}{4}, \ \ \ \mbox{as} \ k \rightarrow \infty.
\end{align}
Combining \eqref{upprhgoal2}, \eqref{lioterm}, \eqref{asyg}, and $\lim_{k} \varpi_{k}=55/12756$, gives that
\begin{align}  \label{upprhgoal3} 
	\frac{k \alpha^{(k)}}{(\frac{1}{2}(\frac{2}{3}+\varpi_{k})- \delta) I(F)} - \frac{k \beta_{1}^{(k)}}{I(F)}-\frac{4k \beta_{2}^{(k)}}{I(F)}
	\leq & \frac{ \int_{0}^{T}tg'(t)^2\:\mathrm{d}t } {(\frac{1}{2}(\frac{2}{3}+\varpi_{k})- \delta)\int_{0}^{T} g(t)^{2} \:\mathrm{d} t} k^{2} + o(k^{2})
	\notag
	\\
	= & \frac{1}{\frac{4}{3}+2\varpi_{k}-2\delta}k^{2}+o(k^{2})
	\notag
	\\
	= & \frac{1}{\frac{4}{3}+2 \cdot \frac{55}{12756}-2\delta}k^{2}+o(k^{2}), \ \ \ \mbox{as} \ k \rightarrow \infty.
\end{align}
Since $\delta$ can be made arbitrarily small, we can take 
\begin{align}  \label{thcon} 
	\rho_{k} = \frac{1}{\frac{4}{3}+2 \cdot \frac{55}{12756}}k^{2}+o(k^{2}) = \frac{2126}{2853} k^{2} + o(k^{2}).
\end{align}
Finally, we remark that the number of integers $\leq x$ that satisfy the requirements of this theorem is $\gg x (\log \log x)^{-1} (\log x)^{-k}$. It can be deduced by using the same argument as in {\cite[Theorem 5.13]{M-A2017}}. The proof of Theorem \ref{NewTh} is now complete.
\end{proof}

\end{document}